\documentclass[a4paper,11pt,english]{elsarticle}

\usepackage[latin1]{inputenc}
\usepackage[english]{babel}
\usepackage{amsmath}
\usepackage{amsthm}
\usepackage{amssymb}
\usepackage{multirow}
\usepackage{url}
\usepackage{color}
\usepackage[colorlinks=true,linkcolor=NavyBlue, citecolor=NavyBlue]{hyperref}

\makeatletter
\newcommand{\addresseshere}{%
  \enddoc@text\let\enddoc@text\relax
}

\usepackage{braket}

\def\F{\mathbf F}

\def\cC{\mathcal C}

\def\cI{\mathcal I}

\def\cS{\mathcal S}

\def\cU{\mathcal U}
\def\cV{\mathcal V}

\def\Sym{\mbox{\rm Sym}}
\def\dim{\mbox{\rm dim}}

\def\FF{{\mathbb F}}

\def\f2{{\mathbb F}_{2}}

\newcommand{\AGL}{\mbox{\rm AGL}}

\newcommand{\GL}{\mbox{\rm GL}}

\newcommand{\ga}{\alpha}

\newcommand{\g}{\gamma}
\newcommand{\gd}{\delta}

\newcommand{\gk}{M}

\newcommand{\gs}{\sigma}

\newcommand{\gt}{\tau}

\newcommand{\be}{{e}}
\newcommand{\bb}{{b}}

\newcommand{\bB}{B}

\newcommand{\Span}{{\mathrm{Span}}}

\def\F2{\mathbb{F}_{\hspace{-0.7mm}2}}
\def\Fp{\mathbb{F}_{\hspace{-0.7mm}p}}
\newcommand\deq{\mathrel{\stackrel{\makebox[0pt]{\mbox{\normalfont\tiny def}}}{=}}}

\newcommand{\T}{\mathcal T}

\newtheorem{theorem}{Theorem}[section]

\newtheorem{lemma}[theorem]{Lemma}

\newtheorem{proposition}[theorem]{Proposition}

\newtheorem{corollary}[theorem]{Corollary}

\theoremstyle{remark}
\newtheorem{remark}[theorem]{Remark}
\newtheorem{example}[theorem]{Example}

\begin{document}

\begin{frontmatter}

\title{On properties of translation groups in the affine general linear group with applications to cryptography\tnoteref{t1}}

\tnotetext[t1]{{This research was partially funded by the Italian Ministry of Education, 
Universities and Research (MIUR), with the project PRIN 2015TW9LSR ``Group theory and applications".
Roberto Civino is partially funded by  the Centre  of Excellence
  EX-EMERGE at University of L'Aquila.}}

\author[no]{Marco Calderini}
\ead{marco.calderini@uib.no}

\author[ro]{Roberto Civino}
\ead{roberto.civino@univaq.it}

\author[tn]{Massimiliano Sala}
\ead{massimiliano.sala@unitn.it}

\address[no]{Department of Informatics, University of Bergen, Norway}
\address[ro]{DISIM, University of l'Aquila, Italy}
\address[tn]{Department of Mathematics, University of Trento, Italy}

\begin{abstract}
The affine general linear group acting on a vector space over a prime field is a well-understood mathematical object. Its elementary abelian regular subgroups have recently drawn attention in applied mathematics thanks to their use in cryptography as a way to hide or detect weaknesses inside block ciphers. This paper is focused on building a convenient representation of their elements which suits better the purposes of the cryptanalyst. Several combinatorial counting formulas and a classification of their conjugacy classes are given as well. 
\end{abstract}


\end{frontmatter}

\section{Introduction}

The group of the translations of a vector space over a prime field is an elementary abelian regular subgroup of the corresponding symmetric group, and its normaliser, 
the affine general linear group, is a well-understood mathematical object. 
Regular subgroups of the affine group and their connections with algebraic structures, such as radical rings~\cite{cds06} and braces~\cite{catino2009regular}, have already been studied in several works~\cite{catino2015regular,hegedus2000regular,pellegrini2016more,pellegrini2017regular}. More recently, elementary abelian regular groups have been used in cryptography to define new operations on the message space of a block cipher and to implement statistical and group theoretical attacks~\cite{brunetta2019hidden,calderini2015differential,civino2019differential}. All these objects are well-known to be conjugated to the translation group, but this fact does not provide a simple description and representation of their elements which is useful to the cryptanalyst. For this reason, we address the problem of giving a convenient matrix representation 
of some elementary abelian regular subgroups of the affine groups and, in some cases, we classify them in terms of their conjugacy classes. The idea behind the cryptographic attack resulting from this work is the one of using alternative group structures on the message space of a block cipher to detect a bias in the distribution of the encrypted messages, as we will describe in the following section in more detail. 
Although the approach of using alternative operations in place of the XOR (the usual sum over a binary vector space) is not new~\cite{abazari2012cryptanalysis,berson1992differential}, the idea of using groups isomorphic to the translation group was never considered. 

\subsection{Organisation of the paper}
The paper is organised as follows. In Section~\ref{sec:construction} we introduce the notation and present the main focus of the work, also providing a description of the idea which is behind the use of translation groups in cryptography. In Section~\ref{sec:2} we present our main result, i.e. Theorem~\ref{th:forma}, which proves a description of the translation groups useful in block ciphers cryptanalysis. Section~\ref{sec:3} is mainly devoted to the case of binary fields, to combinatorial aspects of the topic and to a classification of conjugacy classes in low dimension. In Theorem~\ref{prop:intmax} and Theorem~\ref{prop:bound} we provide a bound on the numbers of groups as in Theorem~\ref{th:forma}.

\section{Preliminaries}\label{sec:construction}

Let us start by introducing the notation used throughout this work.\\

\noindent Let $p$ be a prime number, $n \geq 2$ a positive integer and $V \deq (\Fp)^{n}$ be the $n$-dimensional vector space over $\Fp$.   The $i$-th component of the vector $v \in V$ is denoted by $v^i \in \Fp$. The canonical basis of $V$ is composed by the vectors $\{e_i\}_{i=1}^{n}$, where $e_i^{\,j} = 1$ if and only if $i=j$, otherwise it is 0. 
The vector subspace generated by vectors $v_1,\dots,v_m \in V$ is denoted by $\Span\{v_1,\dots,v_m\}$, where $m\ge1$. Let $\Sym(V)$ be the group of all the permutations on $V$.  In this paper we use postfix notation for function evaluation, i.e. if $g \in \Sym(V)$ and $v \in V$ we write $vg$ to mean $g(v)$.
The identity of $\Sym(V)$ is denoted by $1_{V}$ and if $g_1,\dots,g_m \in \Sym(V)$, where $m\ge 1$, we denote by  $\langle g_1,\dots,g_m\rangle$ the group they generate. Let
 $\GL(V)$ be the general linear group on $V$, i.e. the group of the linear permutations of $V$, and let us denote by $T$ the group of all the translations on $V$, i.e. $T \deq \{\sigma_a \mid a \in V, \sigma_a: V \rightarrow V, x \mapsto x+a\}$. Then, let the affine general linear group $\AGL(V)$, the normaliser of $T$ in the symmetric group, be represented as $\AGL(V) = \GL(V)\ltimes T$.
Let $(\Fp)^{i\times j}$ denote the set of all matrices with entries over $\Fp$ with $i$ rows and $j$ columns. The identity matrix is denoted by $1_{n}$
.\\

 In this work we will also use some basic ring-theoretical notions that are summarised here for the convenience of the reader. 
Let $R$ be a ring. An element $r \in R$ is called \emph{nilpotent} if $r^{\,m}=0$ for some $m\ge 1$ and it is called \emph{unipotent} if $r-1$ is nilpotent, i.e. $(r-1)^m=0$ for some $m\ge 1$.
Analogously, if $H\leq \GL(V)$ is a subgroup of unipotent permutations, then $H$ is called unipotent.
An element $M \in \GL(V)$ is said  \emph{upper unitriangular in a basis} $\{v_1,\dots,v_n\}$ on $V$ if and only if 
$
v_iM-v_i\in \Span\{v_{i+1},\dots,v_{n}\}
$ for all $1\le i\le n$. The map $M$ is called \emph{upper unitriangular} if it is upper triangular with respect to the canonical basis. The group of upper unitriangular linear maps is here denoted by  $\cU(V)$.\\

The idea of the cryptographic application of this study is described in the following section.

\subsection{Motivation and links to the theory of block ciphers}

Let $\mathcal{T} < \Sym(V)$ be elementary abelian regular. As already mentioned, from a result due to Dixon \cite{dixon1971maximal} (see also \cite{aragona2019regular} for an easy proof), there exists $g 	\in \Sym(G)$ such that $\mathcal{T} = T^g\deq g^{-1}Tg$. Since $\mathcal{T}$ inherits from $T$ its regularity, and recalling that for each $a\in V$ we denoted by $\sigma_a \in T$ the translation sending $0$ to $a$, it is possible to represent $\mathcal{T} = \{\tau_a \,|\, a \in V\}$, where the map $\tau_a$ is the unique in $\mathcal{T}$ sending $0$ to $a$. Once this labelling is established, it is possible to define an additive law $\circ$ on $V$ by letting for each $a, b \in V$ $a\circ b\deq a \tau_b$. It is easy to check that $(V,\circ)$ is an abelian group whose corresponding translation group is $T_\circ = \mathcal{T}$. Moreover, letting the multiplication of a vector by a non-zero element $s\in \Fp$ be defined as
\[
s v\deq \underbrace{v\circ\dots\circ v}_{s},
\]
it is easily checked that if $s,t\in \Fp$ and $v,w\in V$, then
\[
s(v\circ w)=s v\circ sw, 
\]
\[
(s+t)v=sv\circ tv,
\]
\[
(st)v=s(tv),
\]
and $pv=0$ since $\mathcal{T}$ is elementary. This proves that $(V,\circ)$ is a vector space over $\Fp$, and since $|V|<\infty$, $(V,\circ)$ and $(V,+)$ are isomorphic vector spaces. We will denote by $\AGL(V,\circ) \deq \AGL(V)^g$ the normaliser of $T_\circ=\mathcal{T}$ and by $\GL(V,\circ)$ the stabiliser of  $\{0\}$ in $\AGL(V,\circ)$. Since in this paper we will always deal with different operations at the same time, for sake of clarity we will sometimes denote $T$ as $T_+$, $\AGL(V)$ by $\AGL(V,+)$ and $\GL(V)$ by $\GL(V,+)$.\\

The idea of using an application of the group-theoretical study of translation groups to block ciphers comes from the fact that the translation is the standard way the user introduces its \emph{key} in the encryption process (in cryptographic terms, the key is XOR-ed to the message). In order to explain this fact and to let the reader figure out the potential attacks coming from alternative translation groups, we will give here a little and self-contained introduction to block ciphers. A \emph{block cipher} on the message space $V$ is a set of many invertible function in $\Sym(V)$, called \emph{encryption functions}. Popular examples may be found e.g. in~\cite{bogdanov2007present,daemen2013design}. Each encryption function is of the type of
\[
\rho\sigma_{k_1}\rho\sigma_{k_2}\ldots\rho\sigma_{k_r},
\] 
where $\rho \in \Sym(V)$ and the parameter $r \in \mathbb N$ are fixed by the designer and made publicly available, and the sequence $(k_1, k_2, \ldots k_r) \in V^r$ represents the \emph{encryption key} chosen by the user. Once the key $(k_1, k_2, \ldots k_r)$ and the message $m \in V$ to be sent are chosen by the sender, it delivers $m\rho\sigma_{k_1}\rho\sigma_{k_2}\ldots\rho\sigma_{k_r}$ to the receiver. If the receiver is entitled to recover the message, i.e. if it knows the secret key, it can apply the inverse of the encryption function and obtain the original message $m$. The security of this process, i.e. the inability of a non-authorised party to recover the message, strongly relies on the way the function $\rho$ is designed. Indeed, the process of designing $\rho$ is one of the most important phases in the definition of a block cipher, and it is usually carried out in order to guarantee that the obtained block cipher is resistant against each known attack (e.g. linear~\cite{matsui1993linear} and differential~\cite{biham1991differential} cryptanalysis). Giving details and properties that the function $\rho$ has to satisfy is out of the scope of this work, for whose purposes is enough to know that a minimum and crucial requirement is that $\rho \notin \AGL(V)$. As a matter of fact, the farthest it lies from the affine group, the better. This guarantees that the group $\langle \,\rho, T \,\rangle$, called \emph{the group of the round functions}, is not the affine group $\AGL(V)$. Such a group, introduced in~\cite{coppersmith1975generators} for the first time, has been carefully studied ever since researchers have shown that some of its properties can reveal weaknesses of the cipher \cite{aragona2019wave,aragona2018primitivity,aragona2017group,aragona2019type,caranti2009some,paterson1636imprimitive,sparr2008group,wernsdorf1992one,wernsdorf2002round}. Although it is rather easy to select $\rho$ such that $\langle\, \rho, T \,\rangle$ is different from $\AGL(V)$, it not as easy to prove that $\langle\, \rho, T\, \rangle$ is not contained in any conjugate of $\AGL(V)$ in $\Sym(V)$. If this is the case, i.e. if there exists $g \in \Sym(V)$ such that  $\langle \,\rho, T \,\rangle < \AGL(V)^g$, then there exists an operation $\circ$ such that 
\begin{equation}\label{eq}
\langle \,\rho, T\, \rangle \leq \AGL(V,\circ),
\end{equation}
which means that each encryption function is affine with respect to the operation $\circ$, a serious threat for the security of the cipher. A description of the attack that can be perfomed in this case is shown in~\cite{calderini2015boolean}. Another example in this regard, i.e. a successful attack against a block cipher which makes use of an operation as described above, can be found in~\cite{civino2019differential}. For the reason explained before, since our interest is in determining if and when the group of the round functions is as in Eq.~\eqref{eq}, we focus on investigating operations $\circ$ such that $T < \AGL(V, \circ)$. Such hypothesis is also decisive in the application studied in~\cite{civino2019differential}, where the classical differential attack (see e.g.~\cite{biham1991differential2,biham1992differential}) is generalised to alternative operations. Moreover, we will always assume $T_\circ < \AGL(V)$, since it guarantees fast computation, crucial in the application to cryptanalysis. The related problem of determining conditions on $\rho \notin \AGL(V)$ which ensure that $\rho \in \AGL(V,\circ)$ for some operation $\circ$ is still open. Some partial results can be found in~\cite{brunetta2019hidden,calderini2015differential,civino2019differential}.\\

In the next section we will introduce our novel results and in particular we will describe all elementary abelian regular groups $T_\circ < \AGL(V,+)$ such that $T_+ < \AGL(V,\circ)$.

\section{Abelian regular subgroups of the affine groups}\label{sec:2}
Keeping in mind the construction of Sec.~\ref{sec:construction}, we now focus on groups conjugated to $T$ which are affine groups. A seminal work for this research is the paper \cite{cds06}, where the authors give an easy description of the abelian regular subgroups of the affine group in terms of commutative associative algebras that one can define on the vector space $(V,+)$. Here we summarise their main results. Recall that a \emph{Jacobson radical ring} is a ring $(V,+,\cdot)$ such that $(V,\diamond)$ is a group, where the operation $\diamond$ defined as $a\diamond b =a + b + a\cdot b$, for each $a, b \in V$. Note that in general the operation $\diamond$ does not induce a vector space structure on $V$.
The proof of the next result may be found in \cite{cds06}.
\begin{theorem}\label{th:somme}
Let $\mathbb{K}$ be any (finite or infinite) field, and $(V,+)$ be a vector space of any dimension over $\mathbb{K}$.
There is a one-to-one correspondence between
\begin{enumerate}
\item abelian regular subgroups of $\AGL(V,+)$, 
\item commutative, associative $\mathbb{K}$-algebra structures $(V,+,\cdot)$ that one can impose on the vector space structure $(V,+)$, such that the resulting ring is radical.
\end{enumerate}
In this correspondence, isomorphism classes of $\mathbb{K}$-algebras correspond to conjugacy classes of abelian regular subgroups of $\AGL(V,+)$, where the conjugation is under the action of $\GL(V,+)$ .
\end{theorem}

\noindent The correspondence mentioned in the previous result may be written explicitly, proceedings as follows. Let $\mathcal{T} < \AGL(V)$ be abelian and regular. Since $\mathcal{T}$ is regular, reasoning as in Sec.~\ref{sec:construction} its elements can be labeled as $\mathcal{T}=\{\tau_a \mid a \in V\}$. For each $a \in V$, from the hypothesis, there exists $M_{a,\mathcal T} \in GL(V,+)$ and $\sigma_b \in T_+$ for some $b \in V$ such that $\tau_{a} = M_{a,\mathcal T}\sigma_b$. In order to keep the notation lighter, $M_{a,\mathcal T}$ will be simply denoted by $M_{a}$. For any $a\in V$, let us define the map $\gd_a\deq M_a-1_V$. Then, operation $\cdot$ defined on $V$ by letting $x\cdot a=x\gd_a$ is such that the structure $(V,+,\cdot)$ is a commutative $\mathbb{K}$-algebra and the resulting ring is radical. Moreover, notice that $0\tau_a = a$ by definition, then $a = 0 \tau_a = 0 M_a\sigma_b=b$, hence $\tau_a = M_a\sigma_a$ for each $a \in V$. 
Denoting by $\circ$ the operation induced by $\mathcal{T}$, let us now define the set

\[
\Omega(\mathcal{T}) = \Omega_\circ \deq \left\{M_a \mid a \in V\right\} < \GL(V),
\]
and denote by $T_\circ = \mathcal T$.
\begin{proposition}\label{lm:uni}
Let $\T < \AGL(V)$ be an elementary abelian regular subgroup. Then for each $a\in V$, $M_a\in \GL(V)$ has order $p$ and it is unipotent. In particular $\Omega(\T)$ is a unipotent subgroup of $\GL(V)$.
\end{proposition}
\begin{proof}
Let $a \in V$. Since $\T$ is elementary, $\tau_a$ has order $p$, so $a\gt_a^{\,\,p-1}=0$. For each $x \in V$ we get
\[
x=x \gt_a^p=(xM_a+a)\gt_a^{\,\,p-1}=(x M_a^2+a\gt_a)\gt_a^{\,\,p-2}=\ldots=xM_a^p+a\gt_a^{\,\,p-1},
\]
therefore $ 0 = M_a^{p}-1_V=\left(M_a-1_V\right)^{p}$.
 \end{proof}
Let us now define an important $V$-subspace:
\[
W(\mathcal{T})\deq \{a\mid\gs_a\in \mathcal{T}\} = \{a\mid\gs_a = \tau_a\}.
\]
We will sometimes denote $W(\mathcal{T})$ by $W_\circ$.  
It is easily  checked that $W(\mathcal{T})$ is a subspace of $(V,+)$ and $(V,\circ)$. Such a subspace is nontrivial for the following theorem, proven in~\cite{cds06}. It is straightforward but important to notice that if $a \in W(\mathcal T)$, then $x+a = x \circ a$ holds for each $x \in V$, and consequently $M_a= 1_n$.

\begin{theorem}[\cite{cds06}]\label{lm:car}
Let $\mathcal{T}\leq \AGL(V,+)$ be an abelian regular subgroup. If $V$ is finite, then $T\cap \mathcal{T} \neq \langle 1_V\rangle$.
\end{theorem}
\noindent We will show soon that  $W(\mathcal{T})$ plays an important role for the characterisation of maps in the group $\mathcal{T}$.\\

Our purpose is, given an operation $\circ$ induced by the group $\mathcal{T} = \{\tau_a \mid a \in V\}$, to describe the matrices $M_a$ for each $a \in V$, where $\tau_a = M_a\sigma_a$. We show now some preliminary results. \\
Let $U$ be a subspace of $V$. Then for all $\g\in\GL(V)$ such that $U\g=U$, the action of $\g$ over $V/U$ is well defined by means of the map $\bar \g:[v]\mapsto [v \g]$ in $\GL(V/U)$.
Let us prove now the following characterisation, recalling that $\cU(V)$ denotes the group of upper unitriangular linear maps.

\begin{lemma}\label{lm:generatori}
Let $M_i \in \cU(V)$ be a unitriangular map acting as the identity on the quotient $V/\Span\{e_{i+1},\ldots,e_{n}\}$, for each $1 \leq i \leq n$. Then, the affine transformations
$M_i\sigma_{e_i}$ generate a transitive subgroup of $\AGL(V)$.

\end{lemma}

\begin{proof}
Denote by $\tau_{e_i}$ the transformation $M_i\sigma_{e_i}$.
Let us start by observing that for each $1\leq i \leq n$ the action of $M_{i}$ over $V/\Span\{e_{i+1},\ldots,e_{n}\}$ is well defined and from the hypotheses $\tau_{e_i}$ acts on vectors of $V$ leaving the first $i-1$ coordinates unchanged. Let now $v =(v^1, v^2, \ldots, v^n)$ and $w =(w^1,w^2, \ldots, w^n)$ be two elements of $V$ and let us show that there exists $\tau \in \langle \tau_{e_1},\tau_{e_2}, \ldots, \tau_{e_n} \rangle$ such that $v\tau=w$. Let $\g^1 \in \Fp$ such that  $v^1 +\g^1 = w^1$. So
\[
v \left(\tau_{e_1}\right)^{\gamma^1}=(w^1,v^2+c^2,\dots,v^n+c^n)\deq v',
\]
for some $c^i \in \Fp$ for $2 \leq i \leq n$, where $c^i$ depends on $v,\tau_{e_1}$ and $\gamma^1$. Analogously, if $\g^2 \in \Fp$ is such that  $\left(v'\right)^2 +\g^2= w^2$, then
\[
v' \left(\tau_{e_2}\right)^{\gamma^2}=(w^1,w^2,v^3+d^3,\dots,v^n+d^n),
\]
for some $d^{\,i} \in \Fp$ for $3 \leq i \leq n$.
In this way, we obtain $$ \tau \deq \left(\tau_{e_1}\right)^{\gamma^1} \left(\tau_{e_2}\right)^{\gamma^2}\cdots \left(\tau_{e_n}\right)^{\gamma^n}$$ such that $v\tau=w$, hence the transitivity is proven.
 \end{proof}

\begin{remark}\label{rm:basis}
Notice that in the conditions of Lemma~\ref{lm:generatori}, if $\circ$ denotes the operation induced by $\mathcal T = \langle \tau_{e_1},\tau_{e_2}, \ldots, \tau_{e_n} \rangle$, then $\{e_i\}_{i=1}^{n}$ is a basis of $(V,\circ)$. However, this is not true in general. In the following example on $V=(\FF_2)^3$ indeed, the canonical basis is not a basis for $(V,\circ)$.
Let $T_\circ$ be defined in the following way:
\[
T_\circ\deq\langle M_{(1,0,1)}\sigma_{(1,0,1)}	, 
M_{(0,1,1)}\sigma_{(0,1,1)}, 
M_{(1,1,1)}\sigma_{(1,1,1)} \rangle,
\]
where 
\[
M_{(1,0,1)}\deq
\begin{pmatrix}
					0&1&1\\
					0&1&0\\
					1&1&0\end{pmatrix},
M_{(0,1,1)} \deq
\begin{pmatrix}
					1&0&0\\
					1&0&1\\
					1&1&0\end{pmatrix} \text{ and }
M_{(1,1,1)} \deq1_n.
\]
Then the translations $\gt_{e_1},\gt_{e_2}, \gt_{e_3}$ are respectively individuated by the matrices
\[
M_{e_1} \deq
\begin{pmatrix}
					1&0&0\\
					1&0&1\\
					1&1&0\end{pmatrix}, 
M_{e_2} \deq
					\begin{pmatrix}
					0&1&1\\
					0&1&0\\
					1&1&0\end{pmatrix} \text{ and }
M_{e_3} \deq
					\begin{pmatrix}
					0&1&1\\
					1&0&1\\
					0&0&1\end{pmatrix}.
\]
It is a straightforward check that $e_1\circ e_2 = e_3$.
\end{remark}
Let us now show a more general result which will be useful later. The following well-known result (see  e.g.~\cite[pag. 62]{waterhouse2012introduction}) is needed.

\begin{theorem}\label{th:unipo}
Let $H\leq \GL(V)$ be a group of unipotent matrices. Then there exists a basis of $V$ in which all elements of $H$ are upper triangular.
\end{theorem}
\begin{lemma}\label{lm:unipotent}
Let $G<\GL(V)$ be a unipotent subgroup and let $U\subseteq V$ be a subspace such that for all $v\in U$ and $g\in G$ we have $v g=v$, i.e. $G$ is a subgroup of the pointwise stabiliser of $U$. Let $d \deq \dim(U)$ and $m\deq n-d$. Then all elements of $G$ are upper triangular in a basis $\{v_1,\dots,v_{m},v_{m+1},\dots,v_{m+d}\}$, where $\{v_{m+1},\dots,v_{m+d}\}$ is any basis of $U$.
\end{lemma}
\begin{proof}
Since $G$ fixes all the elements of $U$, it acts as a group of  unipotent maps on $V/U$. From Theorem \ref{th:unipo} there exists a basis $[v_{1}],\dots,[v_{m}]$ of $V/U$, such that $[v_i]g-[v_i]$ lies in $\Span\{[v_{i+1}],\dots,[v_{m}]\}$ for all $g \in G$. Then, all elements of $G$ are upper triangular in the basis $\{v_1,\dots,v_{m},v_{m+1},\dots,v_{n}\}$, since $v_{i}g-v_{i}=0$ for all $m+1\le i\le n$.
 \end{proof}
\noindent The previous result reads in the way displayed below, when specified to our case. 
\begin{corollary}\label{cor:unipotent}
Let $\T< \AGL(V)$ be an elementary abelian regular group. Let $d \deq \dim (W(\T))$ and let $m\deq n-d$. Then all elements of $\Omega(\T)$ are upper triangular in  a basis $\{v_1,\dots,v_{m+1},\dots,v_{n}\}$, where $\{v_{m+1},\dots,v_{n}\}$ is any basis of $W(\T)$.
\end{corollary}
\begin{proof}
By Proposition~\ref{lm:uni}, $\Omega(\T)$ is unipotent.
Moreover, by definition, for all $v\in W(\T)$ and $M\in \Omega(\T)$ we have $vM=v$. Hence, the claim follows from Lemma \ref{lm:unipotent}. \end{proof}

\noindent The results obtained so far may be summarised in the following theorem. According to this result, when considering an operation $\circ$ we can always assume, up to conjugation, that $W_\circ$ is generated by the last vectors of the canonical basis.

\begin{theorem}\label{cor:identitasotto}
Let $\T< \AGL(V)$ be an elementary abelian regular group. Let $d \deq \dim (W(\T))$ and let $m\deq n-d$. Then there exists $g \in \GL(V)$ such that $\Omega(\T^g) < \cU(V)$ and $W(\T^g)=\Span\{e_{m+1},\dots,e_{n}\}$.
\end{theorem}
\begin{proof}
From Corollary \ref{cor:unipotent}, all the elements of $\Omega(\T)$ are upper triangular with respect to a basis $\{v_1,\dots,v_n\}$ of $V$, whose last $d$ vector form a basis of $W(\T)$. Let $g\in\GL(V)$ such that $v_ig=e_i$ for each $1\leq i \leq n$. It is easy to check that $\Omega(\T^g)=\Omega(\T)^g$, then for all $M \in\Omega(\T)$ we have 
\[
e_ig^{-1}M g-e_i=v_iM g-v_ig=(v_iM-v_i)g.
\]
Since $v_iM-v_i\in \Span\{v_{i+1},...,v_{n}\}$, we have $(v_iM-v_i)g\in \Span\{e_{i+1},...,e_{n}\}$. In conclusion, from $\left(\tau_v\right)^g: x\mapsto x \left(M_v\right)^g g+v g$, we also obtain $W(\T^g)=W(\T)g=\Span\{e_{m+1},\dots,e_{n}\}$.
 \end{proof}
Till now we have assumed that the subgroup $\T$ is an affine group. For reasons already explained in Sec.~\ref{sec:construction} and related to the application in cryptography of this construction, we are interested in groups whose normalisers contain the group of translations $T_+$, i.e. in operations $T_\circ$ for which, given $g \in \Sym(V)$ such that $\T = T_+^g$, we also have $T_+ <  \AGL(V,\circ) =  \AGL(V,+)^g$. Let us report a result from \cite{cds06} which is useful for our purpose. 

\begin{lemma}\label{lm:normalizer}
Let $\T<\AGL(V)$ be abelian and regular. Then for each $\gs_x\in T_+$  and $\gt_y\in \T$ we have
\[
[\gs_x,\gt_y]=\gs_{x\cdot y},
\]
where $\cdot$ denotes the product of the $\Fp$-algebra related to $\T$ as in Theorem \ref{th:somme}, and $[\gs_x,\gt_y]\deq \gs_x^{-1}\gt_y^{-1}\gs_x\gt_y$. 
\end{lemma}

\noindent In our case, from Lemma \ref{lm:normalizer} we obtain that $T$ normalises $\T < \AGL(V)$ if and only if $\gs_{x\cdot y}\in \T$ for all $x, y \in V$. Indeed, if for all $\gs_x \in T$ we have $\T^{\gs_x}=\T$, then
\[
\gs_{x\cdot y}={\gs_x^{-1}\gt_y^{-1}\gs_x}\gt_y\in \T.
\]
Conversely, if $\gs_{x\cdot y}\in \T$ for each $x,y \in V$, then 
\[
\T \ni \gs_{x\cdot y}\gt_y^{-1}=\gs_x^{-1}\gt_y^{-1}\gs_x.
\]
Finally notice that the condition $\gs_{x\cdot y}\in \T$ for all $x, y \in V$ is equivalent to $x\cdot y\cdot z=0$ for all $x,y,z \in V$.\\

We are now ready to prove one of the main results of this work, i.e. the structure of affine translation groups whose normalisers contain the group $T_+$. Before doing so, let us recall that for sake of simplicity, proceeding as in Sec.~\ref{sec:construction}, given a group $\T = T_\circ < \AGL(V)$, we denote by $\AGL(V,\circ)$ the normaliser in $\Sym(V)$ of $\T$, which is $\AGL(V,+)^g$ where $g \in \Sym(V)$ is such that $\T = T^g$.  

\begin{theorem}\label{th:forma}
Let $\T< \AGL(V,+)$ be elementary abelian regular and let $\circ$ be the operation induced on $V$. Let $d \deq \dim (W(\T))$, let $m\deq n-d$ and let us assume $W(\T)=\Span\{e_{m+1},\dots,e_{n}\}$. Then, $T_+ < \AGL(V,\circ)$ if and only if for all $M_y\in \Omega(\T)$ there exists a matrix $B_y\in (\Fp)^{m\times d}$ such that
\begin{equation}\label{eq:goal}
M_y=\begin{pmatrix}
1_{m} & B_y\\
0 & 1_{d} \end{pmatrix}.
\end{equation}
\end{theorem}
\begin{proof}
By Theorem~\ref{cor:identitasotto}, there exists another group operation $\diamond$ on $V$ such that the corresponding translation group is conjugated, by an element of $\GL(V)$, to $T_\circ$ and satisfies $W(T_\diamond)= W(T_\circ)$ and $\Omega(T_\diamond) = \{\overline{M_a} \mid a \in V\}<\cU(V)$. Let $y \in V$ and let $A_y\in (\Fp)^{m\times m}$ an upper-triangular matrix and $B_y\in (\Fp)^{m\times d}$ such that
\[
\overline{M_y}=\begin{pmatrix}
A_y & B_y\\
0 & 1_{d} \end{pmatrix}.
\]
Notice that the lower structure of the matrix derives by the property $e_i \in W(T_\diamond)$ for each $m+1 \leq i \leq n$, i.e. 
$y \diamond {e}_i=e_i\overline{M}_y+y =y+{e}_i$ for each $m+1 \leq i \leq n$. Recall that 
\begin{eqnarray}
T_+ < \AGL(V,\diamond) &\iff& \forall x,y \in V \quad x \cdot y \in W(T_\diamond) \label{eq:one_main}\\
 &\iff & \forall x,y \in V \quad x\overline{M_y} - x \in W(T_\diamond), \label{eq:two_main}
\end{eqnarray}
where the equivalence in Eq.\eqref{eq:one_main} derives from Lemma~\ref{lm:normalizer}. From Eq.\eqref{eq:two_main} instead, considering $x\in \Span\{e_{1},\dots,e_{m}\}$
we obtain that $x\overline{M_y}-x \in W(T_\diamond)$ if and only if $A_y=1_{m}$. 

In order to conclude, we need to prove that each conjugate $T_\circ = {T_\diamond}^{g}$ is such that all the matrices in the group $\Omega(T_\circ)$ are as in Eq.~\eqref{eq:goal}, provided that $g \in \GL(V)$ and $W(T_\circ)$ is spanned by the last $d$ vectors of the canonical basis. Let $g \in \GL(V)$ such that $T_\circ = T_\diamond^{\,g}$.
Since $W({T_\diamond}){g} = W({T_\diamond}^{g})=W(T_\circ)$, then $\Span\{e_{m+1},\dots,e_{n}\}g=\Span\{e_{m+1},\dots,e_{n}\}$ and also $\Span\{e_{m+1},\dots,e_{n}\}g^{-1}=\Span\{e_{m+1},\dots,e_{n}\}$. Consequently
\[
g=\begin{pmatrix}{}
G_1&G_2\\
0&G_3\end{pmatrix} \text{ and } g^{-1}=\begin{pmatrix}
G_1^{-1}&{G_2}'\\
0&G_3^{-1}\end{pmatrix},
\]
for some $G_1\in(\Fp)^{m\times m},G_2,{G_2}'\in(\Fp)^{m\times d}$ and $G_3\in(\Fp)^{d\times d}$.
Thus, if $M\in\Omega(T_\diamond)$ we have
\[
M^g=\begin{pmatrix}
G_1^{-1}&{G_2}'\\
0&G_3^{-1}\end{pmatrix} 
\begin{pmatrix}
1_{m} & B_{m\times d}\\
0 & 1_{d} \end{pmatrix}
\begin{pmatrix}
G_1&G_2\\
0&G_3\end{pmatrix}
=
\begin{pmatrix}
1_{m} & {{B'}_{m\times d}}\\
0 & 1_{d} \end{pmatrix},
\]
therefore the claim follows from $\Omega(T_\circ)=\Omega({T_\diamond}^g)=\Omega(T_\diamond)^g$.
 \end{proof}

\noindent The characterisation given above allows to construct an isomorphism between the vector spaces $(V,\circ)$ and $(V,+)$, which can be computed very efficiently (see \cite{calderini2015boolean}). This makes some attacks feasible~\cite{calderini2015boolean,civino2019differential}.
Moreover, Theorem~\ref{th:forma} can be used to determine the maps contained in $\GL(V,\circ)\cap\GL(V,+)$ (see \cite{brunetta2019hidden,civino2019differential}).

\section{Even characteristic and combinatorial formulas}\label{sec:3}
In this section we specialise our focus to the cryptographically-relevant case of binary fields. Let us assume from now on that $p=2$. 
In this case, we can prove (see Theorem~\ref{prop:intmax} and Theorem~\ref{prop:bound}) an upper bound on the number of the elementary abelian regular subgroups as in Theorem~\ref{th:forma}. Moreover, we can calculate the number of these groups if the co-dimension of $W(T_\circ)$ is 2 or 3. To conclude, we report the full classification of the elementary abelian regular subgroups of $\AGL(V,+)$ up to dimension 6. Before doing so, let us prove the following result which bounds the dimension of the subspace $W(T_\circ)$.

\begin{proposition}\label{prop:intmax}
Let $\T< \AGL(V,+)$ be elementary abelian regular and let $d \deq \dim (W(\T))$. 
If $\T\ne T$, then \[\dfrac{(-1)^n+3}{2}\le d\le n-2.\]
\end{proposition}
\begin{proof}
From Theorem~\ref{lm:car} and from the hypothesis  we have $1 \leq d \leq n-1$. Let us now assume that $W(\T)$ cointains $n-1$ linearly independent vectors $v_1,v_2\dots,v_{n-1} \in V$ and let $v_n \in V$ independent from $v_1,\dots,v_{n-1}$. Let $\circ$ be the operation induced by $\T$.
Then, $v_i\circ v_n=v_i +v_n$, thus  $v_iM_{v_n}=v_i$ for all $1\le i \le n-1$. Moreover, $v_n\circ v_n=0$ and so $v_nM_{v_n}=v_n$. Then, if $v\in V$, then 
\[
v\circ v_n=\left(\sum_{i<n}\xi_iv_i+\xi_nv_n\right)M_{v_n}+v_n=\sum_{i<n}\ga_iv_i+\ga_nv_n+v_n=v+v_n,
\] 
which implies $d=n$, a contradiction. If $n$ is even, then $d>1$, i.e. $T \cap \T$ contains at least four elements. A proof of this fact may be found in~\cite{brunetta2019hidden}.
 \end{proof}

Let us now prove that if $\T$ normalises $T$ and the co-dimension of $W(\T)$ is at most $5$, then we also have that $T$ normalises $\T$.
\begin{proposition}\label{lm:forma}
Let $\T< \AGL(V)$ be elementary abelian regular, and let $\circ$ be the operation induced. Let $d \deq \dim (W(\T))$ and $m\deq n-d$. If $2\le m\le 5$, then $\AGL(V,\circ)$ contains $T$.
\end{proposition}
\begin{proof}
The claim follows if we prove that if $x,y \in V$, then $x \cdot y \in W(\T)$. Let $x, y \in V$ and let us assume by contradiction $x \cdot y \notin W(\T)$.
 Then there exists $z \notin W(\T)$ such that $x\cdot y\cdot z \ne 0$. Let us show that $x, y,z,x\cdot y, x\cdot z, y \cdot z, x\cdot y\cdot z$ are linearly independent. Let $\xi_i \in \mathbb{F}_2$ for $1 \leq i \leq 7$ such that
\[
\xi_1x +\xi_2y+\xi_3z+\xi_4x\cdot y+\xi_5x\cdot z+\xi_6y\cdot y+\xi_7x\cdot y\cdot z = 0.
\]
By multiplying each member of the previous equation by $y\cdot z$ we obtain $\xi_1 x\cdot y\cdot z = 0$, which implies $\xi_1 = 0$. In the same way, by multiplying  by $x \cdot z$ we prove $\xi_2 = 0$. Proceeding in this way one proves that $\xi_i = 0$ for each $1 \leq i \leq 7$.  This proves that $x, y,z,x\cdot y, x\cdot z, y \cdot z, x\cdot y\cdot z$ are linearly independent and none of these belongs to $W(\T)$. Using a similar argument one proves that $\Span\{x, y,z,x\cdot y, x\cdot z, y \cdot z, x\cdot y\cdot z\}\cap W(\T) = \{0\}$. This implies $m\geq 6$, a contradiction.
 \end{proof}
We have presented the previous result in the way which best fit our needs. However, it can be stated more generally in the following way.
\begin{proposition}
Let $\T_1, \T_2 < \Sym(V)$ be elementary abelian regular. Let $d$ be such that $2^d =|\T_1 \cap \T_2|$, $m\deq n-d$ and let us assume $2\le m\le 5$. Then $\T_1$ is contained in the normaliser of $\T_2$ if and only if $\T_2$ is contained in the normaliser of $\T_1$.
\end{proposition}
\begin{example}
Notice that Proposition \ref{lm:forma} does not hold, in general, for $m\ge 6$. Let $(V,+,\cdot)$ be the exterior algebra over a vector space of dimension three, spanned by $e_1,e_2,e_3$. Hence a basis of $V$ is composed by
\[
e_1,e_2,e_3,e_4=e_1\wedge e_2, e_5=e_1\wedge e_3, e_6=e_2\wedge e_3, e_7=e_1\wedge e_2 \wedge e_3.
\]
The associated translation group $T_\circ$ is such that $W(T_\circ)=\Span\{e_7\}$, but we have
\[
M_{e_1}=\begin{pmatrix}
1& 0& 0& 0& 0& 0& 0\\
 &1 &0 &\fbox{1} &0 &0 &0\\
 & &1 &0 &\fbox{1} &0 &0\\
 & & &1 &0 &0 &0\\
 & & & &1 &0 &0\\
 && & & & 1& 1\\
&&  & & & & 1\end{pmatrix}.
\]
From Theorem \ref{th:forma}, $\AGL(V,\circ)$ cannot contain the group $T_+$.
\end{example}

Let us now point out, starting from Theorem~\ref{th:forma}, some properties of the matrices in $\Omega(T_\circ)$ defining the operation $\circ$. Let us assume $T_\circ < \AGL(V)$ be elementary abelian regular and let us denote, as usual, $d \deq \dim (W(T_\circ))$ and $m\deq n-d$. Let $1 \leq i \ne j \leq m$.  Since $e_i \circ e_i = e_i M_{e_i} + e_i = 0$ we obtain that the $i$-th row of $B_{e_i}$ is zero, where
\[
M_{e_i}=\begin{pmatrix}
1_{m} & B_{e_i}\\
0 & 1_{d} \end{pmatrix}.
\]
Instead, from $e_i \circ e_j = e_iM_{e_j} + e_j = e_jM_{e_i} + e_i = e_j \circ e_i$, we obtain that the $j$-th row of $B_{e_i}$ equals the $i$-th row of $B_{e_j}$. Moreover,  let $x \in V$. Then 
\begin{eqnarray*}
x \circ e_i \circ e_j &=&\left(xM_{e_i}+e_i\right)\circ e_j\\
 &=& \left(xM_{e_i}+e_i\right)M_{e_j} + e_j\\
  &=& xM_{e_i}M_{e_j}+e_iM_{e_j}+e_j\\
  &=& xM_{e_i}M_{e_j} + e_i\circ e_j,
\end{eqnarray*}
which proves that $M_{e_i\circ e_j} = M_{e_i}M_{e_j}$, i.e. 
\[
M_{e_i\circ e_j}=\begin{pmatrix}
1_{m} & B_{e_i}+B_{e_j}\\
0 & 1_{d} \end{pmatrix}.
\]
This fact is easily generalised as follows. 

\begin{proposition}\label{rm:comb}
Let $T_\circ< \AGL(V)$ be an elementary abelian regular group. Let $d \deq \dim (W(T_\circ))$ and $m\deq n-d$. Moreover, let us assume $W(T_\circ)=\Span\{e_{m+1},\dots,e_{n}\}$ and $T < \AGL(V,\circ)$.
Let $x\in V$, $x=\xi_1 e_1\circ\dots\circ \xi_{n}e_{n}$ for some $\xi_i \in \FF_2$. Then 
\[
M_{x}=\begin{pmatrix}
1_{m} & \sum_{i=1}^m\xi_i B_{e_i}\\
0 & 1_{d} \end{pmatrix}.
\]
\end{proposition}
\begin{proof}
From the hypothesis we have that the canonical basis of $(V,+)$ is a basis also for $(V,\circ)$ (see Remark \ref{rm:basis}). Moreover, $B_{e_i} \ne 0$ for $1\leq i \leq m$ and $B_{e_i} = 0$ for $m \leq i \leq n$. The claim follows straightforwardly by writing $x$ in terms of $e_i$s in $(V,\circ)$.
 \end{proof}
\subsection{Some combinatorial results}
In this section we will examine some combinatorial aspects of our topic, focusing on counting the number of abelian regular subgroups of the affine group which are useful in cryptographic contexts. In the next result we will count them in terms of points of a given geometric variety. Let $T_\circ$ be as in Proposition~\ref{rm:comb}. For each $1 \leq i \leq m$ we will denote the entries in the matrix $M_{e_i}$ in the following way:
\begin{equation}\label{eq:gammai}
M_{e_i}=\left(\begin{array}{cccc}
 & b^{(i)}_{1,1}&\dots&b^{(i)}_{1,d}\\
 1_{m}  & \vdots& &\vdots\\

 & b^{(i)}_{m,1}&\dots&b^{(i)}_{m,d}\\
 &&&\\
0 & &1_{d}& \\
 &&&\end{array}\right).
\end{equation}
In what follows, in order to keep the notation more compact, given a positive integer $s$ we will denote by $[s]$ the set $\{1,\ldots,s\}$.

\begin{theorem}\label{lm:numero}
Let $d \geq 1$.
The number of elementary abelian regular subgroups $T_\circ < \AGL(V,+)$ such that $\dim (W(T_\circ))=d$ and $T_+ < \AGL(V,\circ)$ is
\begin{equation}\label{eq:numero}
{n\brack d}_2\cdot |\cV(\cI_{m,d})|,
\end{equation}
where $m= n-d$, $\cI_{m,d}$ is the ideal in $\FF_2\left[b^{(s)}_{i,j}\middle| {i,s\in[m],\, j\in[d]}\right]$ generated by 
$
\cS_0\cup\cS_1\cup\cS_2\cup\cS_3
$
with
\[
\begin{aligned}
\cS_0&\deq\left\{\left(b^{(s)}_{i,j}\right)^2-b^{(s)}_{i,j}\middle | i,s\in [m],j\in[d]\right\},\\
\cS_1&\deq\left\{\prod_{i=1}^{m}\prod_{j=1}^{d}\left(1+\sum_{s\in S}b^{(s)}_{i,j}\right)\middle | S\subseteq [m],S\ne \emptyset\right\},\\
\cS_2&\deq\left\{b^{(s)}_{i,j}-b^{(i)}_{s,j}\middle| i,s\in [m],j\in[d]\right\},\\
\cS_3&\deq\left\{b^{(i)}_{i,j}\middle| i\in [m],j\in[d]\right\},
\end{aligned}
\]
 $\cV(\cI_{m,d})$ is the variety of $\cI_{m,d}$ and ${n\brack d}_2\deq\prod_{i=0}^{d-1}\frac{2^{n-i}-1}{2^{d-i}-1}$ is the Gaussian binomial.
\end{theorem}
\begin{proof}
The claim follows by applying together Theorem~\ref{th:forma} and Theorem~\ref{cor:identitasotto}.
Let us start by computing the number of the groups as in Theorem \ref{th:forma}, and then all the conjugates one can obtain from these. Notice that a group $T_\circ < \AGL(V,+)$ such that $W(T_\circ)$ is generated by the last $d$ vectors of the canonical basis of $V$ and such that $T_+ < \AGL(V,\circ)$ is determined if the matrices $M_{e_1},\dots,M_{e_m}$ (and so, equivalently, $B_{e_1},\dots,B_{e_m}$) are individuated, since $M_{e_i} = 1_n$ for the remaining $m < i \leq n$.
We will show that to each set of admissible matrices $\{B_{e_1},\dots,B_{e_m}\}$ corresponds one point in $\cV(\cI_{m,d})$ and vice versa, from a point of $\cV(\cI_{m,d})$ we can obtain one set of admissible matrices $\{B_{e_1},\dots,B_{e_m}\}$.
Let $T_\circ < \AGL(V,+)$ be such that $W(T_\circ)$ is generated by the last $d$ vectors of the canonical basis of $V$ and such that $T_+ < \AGL(V,\circ)$. Let us denote by $\{M_{e_1},\dots,M_{e_m}\}$ the matrices defining the operation. If $\emptyset \ne S \subseteq [m]$ and $x=\displaystyle\underset{i\in S}{\bigcirc}e_i$, then, from Proposition~\ref{rm:comb},
\[
M_x=\left(\begin{array}{cccc}
 & \sum_{s\in S}b^{(s)}_{1,1}&\dots& \sum_{s\in S}b^{(s)}_{1,d}\\
 1_{m}  & \vdots& &\vdots\\

 &  \sum_{s\in S}b^{(s)}_{m,1}&\dots& \sum_{s\in S}b^{(s)}_{m,d}\\
  &&&\\
 0& &1_{d}& \\
 &&&\end{array}\right).
\]
Since $M_x\ne 1_V$, then there exist $i,j$ such that 
\[
\displaystyle\sum_{s\in S}b^{(s)}_{i,j}=1,
\]
which happens if and only if
\[
\prod_{i=1}^{m}\prod_{j=1}^{d}\left(1+\sum_{s\in S}b^{(s)}_{i,j}\right)=0.
\]
For simmetry we also have that the conditions given by set $\cS_2$ hold. Moreover, since $e_i$ is fixed from $M_{e_i}$, we also obtain a solution for set $\cS_3$. To conclude, $\cS_0$ is trivially satisfied, since the matrices are binary.\\
Vice versa, from a solution of the ideal $\cI_{m,d}$, we can construct $B_{e_1},\dots,B_{e_m}$  as in Eq.~\eqref{eq:gammai}. Consequently, we can consider the group $\T$ generated by the affine maps $\gt_{e_i}\deq M_{e_i}\gs_{e_i}$ for $1 \leq i \leq n$, where for $1\le i\le m$
\[
M_{e_i}\deq \begin{pmatrix}
1_{m} & B_{e_i}\\
0 & 1_{d} \end{pmatrix} 
\]
 and $M_{e_i}\deq 1_{n}$ for $m< i\le n$.
Since the conditions of Lemma \ref{lm:generatori} are satisfied, $\T$ is transitive, and it is abelian from the condition expressed by set $\cS_2$.
Moreover, if $x\in V$ and $1\leq i\leq m$, then
\[
x\gt_{e_i}^2=(xM_{e_i}^2+e_iM_{e_i}+e_i).
\]
Computing $M_{e_i}^2$ we obtain 
\[
M_{e_i}^2=\begin{pmatrix}
1_{m} & B_{e_i}+B_{e_i}\\
0 & 1_{d} \end{pmatrix}=1_{n}.
\]
Hence,  since from the condition given by set $\cS_3$ we obtain $e_iM_{e_i}=e_i$, and so $\gt_{e_i}^2=1_V$, i.e. $\T$ is elementary. Moreover, $\T$ is regular, 
since it is abelian and transitive.\\
This shows a one-to-one correspondence between the points of $\cV(\cI_{m,d})$ and the subgroups $T_\circ < \AGL(V,+)$ such that $W(T_\circ)=\Span\{e_{m+1},\dots,e_{n}\}$ and $T_+ <\AGL(V,\circ)$.
To conclude, consider a $d$-dimensional vector subspace $\overline{W} < V$ and let $\Delta=|\cV(\cI_{m,d})|$. Let us denote by $\T_1,\dots,\T_\Delta$ the distinct elementary abelian regular groups such that $W(\T_i)=\Span\{e_{m+1},\dots,e_{n}\}$ and let $g\in\GL(V,+)$ be a transformation such that $\overline{W}g=\Span\{e_{m+1},\dots,e_{n}\}$.
Then the groups $(\T_1)^{g^{-1}},\dots,(\T_\Delta)^{g^{-1}}$ are pairwise distinct and $W((\T_i)^{g^{-1}})=\overline{W}$ for each $1 \leq i \leq \Delta$. Now, let $T_\diamond$ be an elementary abelian regular subgroup such that $W(T_\diamond)=\overline{W}$. 
We have $W((T_\diamond)^g)=W(T_\diamond)g=\Span\{e_{m+1},\dots,e_{n}\}$, which implies $(T_\diamond)^g=\T_i$ for some $i$, and so $T_\diamond=(\T_i)^{g^{-1}}$. Our claim follows from the fact that the number of $d$-dimensional vector subspaces of an $n$-dimensional vector space over $\mathbb{F}_2$ is ${n\brack d}_2$.
 \end{proof} 
In the next result, we give an upper bound on the number of points of the variety $\cV(\cI_{m,d})$ defined in Theorem~\ref{lm:numero}. A lower bound to $|\cV(\cI_{m,d})|$ has been given in~\cite{brunetta2019hidden}, where it is also shown that the upper bound of Theorem~\ref{prop:bound} is tight.

\begin{theorem}\label{prop:bound}
Let $\cI_{m,d}$ be defined as in Theorem \ref{lm:numero}. Then
\[
 |\cV(\cI_{m,d})|\le 2^{d\frac{m(m-1)}{2}}-1-\sum_{r=1}^{m-2}{m\choose r}\left(2^{d}-1\right)^{\binom{m-r}{2}}.
\]
\end{theorem}
\begin{proof}
Let
$
\overline{B}=(b_1^{(1)},\dots,b_m^{(1)},b_1^{(2)},\dots,b_m^{(2)},\dots,b_1^{(m)},\dots,b_m^{(m)})\in\cV(\cI_{m,d}),
$
where $b_i^{(s)}=(b_{i,1}^{(s)},\dots b_{i,d}^{(s)})\in (\FF_2)^d$ for all $i,j$ as in \eqref{eq:gammai}, i.e. $b_i^{(s)}$ is the $i$-th row of the matrix $B_{e_s}$. \\
We aim at counting how many vectors $\overline{B}$ satisfy the constrains of set $\cS_1$, $\cS_2$ and $\cS_3$ as in Theorem~\ref{lm:numero}. We proceed in two steps: 
we consider first all the solutions for $\cS_2$ and $\cS_3$ and then we exclude some of those for which the equations of $\cS_1$ are not satisfied.\\

{\bf 	\noindent First step.} As already pointed out before Proposition~\ref{rm:comb}, from the conditions in $\cS_3$ we have $b_i^{(i)}=0$ for all $i$, and from those in $\cS_2$, $\bb_j^{(i)}=\bb_i^{(j)}$ for all $i,j$. Therefore, the matrix $B_{e_1}$ is determined only by the rows $\bb_2^{(1)},\dots,\bb_m^{(1)}$,  being its first row equal to zero. Analogously, $B_{\be_2}$ is determined only by the rows $\bb_3^{(2)},\dots,\bb_m^{(2)}$ and by $\bb_2^{(1)}$, since the first row of $B_{\be_2}$ is equal to the second row of $B_{\be_1}$ and since the second row of $B_{\be_2}$ equal to zero. Iterating this argument we can consider only the vector composed as
\[
B=(\underbrace{\bb_2^{(1)},\dots,\bb_m^{(1)}},\underbrace{\bb_3^{(2)},\dots,\bb_m^{(2)}},\dots,\underbrace{\bb_{m-1}^{(m-2)},\bb_m^{(m-2)}},\underbrace{\bb_m^{(m-1)}})
\]
and thus we have $2^{d\frac{m(m-1)}{2}}$ solutions to the equations in $\cS_2\cup\cS_3$.\\

{\bf \noindent Second step.} The entries of $\bB$ must satisfy also the constrains given by $\cS_1$, so for any subset $S 	\subset [m]$ we can exclude the cases where
 \[
\begin{cases}
B_{\be_i}=0 &\text{if $i \in S$}\\
B_{\be_i}\ne 0 &\text{if $i \notin S$}.
 \end{cases}
 \]
In particular, we count when the entries of the matrices $B_{\be_i}$  with $i\in S$ are all zeros and the remaining entries of the matrices $B_{\be_i}$ with $i\notin S$ are all non-zero.
We start considering those vectors $\bB$ obtained when exactly one $B_{\be_i}$ is zero and others are non-zero, that is, we consider any set $S$ with one element. In this case $n-1$ entries  of $\bB$ are zero and the others are all non-zero. Similarly, if any pair $(B_{\be_s},B_{\be_t})$ is equal to zero and the others are not, then $m-1+m-2$ entries of $\bB$ are zero and the others are all non-zero. Indeed, assuming $s<t$, the zero entries of $\bB$ must be $\bb_s^{(1)},...,\bb_s^{(s-1)}, \bb_{s+1}^{(s)}, ...,\bb_m^{(s)}$ in order to have $B_{\be_s}=0$, and $\bb_t^{(1)},...,\bb_t^{(t-1)}, \bb_{t+1}^{(t)}, ...,\bb_m^{(t)}$ in order to have  $B_{\be_t}=0$. Considering that $\bb_t^{(s)}$ is already zero, we have that $m-1+m-2$ entries of $\bB$ are zero. Iterating this argument, if we assume that $r$ matrices are zero, then $\sum_{i=1}^r(m-i)$  entries of $\bB$ are zero and the others are all non-zero. Then such $r$ matrices can be chosen in $\binom{m}{r}$ possible ways and any time $2^d-1$ non-zero elements may be used to fill each of the other entries of $\bB$, that are 
\[
\begin{aligned}
\frac{m(m-1)}{2}-\sum_{i=1}^r (m-i)&=\binom{m}{2}-\sum_{i=m-r}^{m-1} i\\
&=\binom{m}{2}-\sum_{i=1}^{m-1} i+\sum_{i=1}^{m-r-1} i\\
&=\binom{m}{2}-\binom{m}{2}+\binom{m-r}{2}\\
&=\binom{m-r}{2}.
\end{aligned}
\]
The last case is when $m-1$ matrices $B_{\be_i}$ are zero. By the conditions of $\cS_2\cup\cS_3$ also the last one is zero, and this happens only when $\bB$ is zero.
This concludes the proof.
 \end{proof}

The following results are derived from Theorem~\ref{lm:numero} and are related to the special cases when $\dim(W(T_\circ)) \in \{n-3, n-2\}$. Notice that the case $\dim(W(T_\circ)) = n-2$ has been largely considered in \cite{civino2019differential}, where it has been used to perform a differential attack against a block cipher. The same notation as in Theorem~\ref{lm:numero} in used. Recall that if $\T = T_\circ$, from Proposition~\ref{lm:forma}, the hypothesis $T_\circ < \AGL(V,+)$ is enough to guarantee that  $T_+ < \AGL(V,\circ)$, and so also Theorem~\ref{th:forma} applies.

\begin{corollary}\label{cor:cod3}
There exist \[{n\brack n-3}_2 \cdot \left(2^{3(n-3)}-7(2^{n-3}-1)-1\right)\] distinct elementary abelian regular groups $\T < \AGL(V)$ such that $\dim(W(\T))=n-3$.
\end{corollary}
\begin{proof}
Proceeding as in Theorem \ref{lm:numero}, we need to compute the number of groups $\T$ such that $W(\T)=\Span\{\be_4,\dots,\be_n\}$.
Using the notation as in Theorem~\ref{prop:bound}, we have

\begin{eqnarray*}
\gk_{\be_1}=\left(\begin{array}{ccc|c}
1&0&0&0\\
 &1&0&\bb^{(1)}_2\\
 & &1&\bb^{(1)}_3\\
\hline
  & & &  1_{n-3}  \end{array}\right)
,&
  \gk_{\be_2}=\left(\begin{array}{ccc|c}
1 &0&0&\bb^{(1)}_2\\
 &1&0&0\\
 & &1&\bb^{(2)}_3\\
\hline
  & & &  1_{n-3} 
\end{array}\right)& \\
\gk_{\be_3}=\left(\begin{array}{ccc|c}
1& 0&0&\bb^{(1)}_3\\
 &1&0&\bb^{(2)}_3\\
 & &1&0\\
 \hline
  & & & 1_{n-3}  \end{array}\right)
,&
  \gk_{\be_1}  \gk_{\be_2}=\left(\begin{array}{ccc|c}
1 &0&0&\bb^{(1)}_2\\
 &1&0&\bb^{(1)}_2\\
 & &1&\bb^{(1)}_3+\bb^{(2)}_{3}\\
 \hline
  & & & 1_{n-3}  \end{array}\right)
 & \\
   \gk_{\be_1}  \gk_{\be_3}=\left(\begin{array}{ccc|c}
1 &0&0&\bb^{(1)}_{3}\\
 &1&0&\bb^{(1)}_{2}+\bb^{(2)}_{3}\\
 & &1&\bb^{(1)}_{3}\\
\hline
  & & &  1_{n-3} \end{array}\right)
,&
  \gk_{\be_2}  \gk_{\be_3}=\left(\begin{array}{ccc|c}
1 &0&0&\bb^{(1)}_{2}+\bb^{(1)}_{3}\\
 &1&0&\bb^{(2)}_{3}\\
 & &1&\bb^{(2)}_{3}\\
\hline
 & & & 1_{n-3}  \end{array}\right)&
 \\
  \gk_{\be_1}   \gk_{\be_2}  \gk_{\be_3}=\left(\begin{array}{ccc|c}
1 &0&0&\bb^{(1)}_{2}+\bb^{(1)}_{3}\\
 &1&0&\bb^{(1)}_{2}+\bb^{(2)}_{3}\\
 & &1&\bb^{(1)}_{3}+\bb^{(2)}_{3}\\
 \hline
  & & & 1_{n-3} \end{array}\right).
\end{eqnarray*}
The following possibilities need to be ruled out:
\begin{enumerate}
\item $\gk_{\be_1}=1_n$ $\Leftrightarrow$ $\bb^{(1)}_2=0$ and $\bb^{(1)}_3=0$,
\item $\gk_{\be_2}=1_n$ $\Leftrightarrow$  $\bb^{(1)}_2=0$ and $\bb^{(2)}_3=0$,
\item $\gk_{\be_3}=1_n$ $\Leftrightarrow$  $\bb^{(1)}_3=0$ and $\bb^{(2)}_3=0$,
\item $  \gk_{\be_1}   \gk_{\be_2}=1_n$ $\Leftrightarrow$  $\bb^{(1)}_2=0$ and $\bb^{(1)}_3=\bb^{(2)}_3$,
\item $  \gk_{\be_1}   \gk_{\be_3}=1_n$ $\Leftrightarrow$ $\bb^{(1)}_3=0$ and $\bb^{(1)}_2=\bb^{(2)}_3$,
\item $\gk_{\be_2}  \gk_{\be_3}=1_n$ $\Leftrightarrow$ $\bb^{(1)}_2=\bb^{(1)}_3$ and $\bb^{(2)}_3=0$,
\item $  \gk_{\be_1}   \gk_{\be_2}  \gk_{\be_3}=1_n$ $\Leftrightarrow$ $\bb^{(1)}_2=\bb^{(1)}_3$, $\bb^{(1)}_2=\bb^{(2)}_3$ and $\bb^{(1)}_3=\bb^{(2)}_3$.
\end{enumerate}
Therefore we obtain that $2^{3(n-3)}-7(2^{n-3}-1)-1$ is the number of distinct subgroups $\T$ such that $W(\T)=\Span\{\be_4,\dots,\be_n\}$.
\end{proof}
\begin{corollary}\label{cor:cod2}
There exist $${n \brack n-2}_2 \cdot (2^{n-2}-1)$$ distinct elementary abelian regular groups $\T < \AGL(V)$ such that $\dim(W(\T))=n-2$.
\end{corollary}
\begin{proof}
The proof is obtained using the same argument as in Corollary~\ref{cor:cod3}.
 \end{proof}
Let us now prove that the groups of Corollary~\ref{cor:cod2} belong to the same conjugacy class under $\GL(V)$.
\begin{proposition}
Let $\T$ and $\T'$ elementary abelian regular subgroups of $\AGL(V,+)$ such that $\dim(W(\T))=\dim(W(\T'))=n-2$. Then, there exists $g \in \GL(V)$ such that $\T' = \T^g$.
\end{proposition}
\begin{proof}
It is enough to prove the claim for $\T$ and $\T'$ elementary abelian regular subgroups of $\AGL(V,+)$ such that $W(\T)=W(\T')=\Span\{\be_3,\dots,\be_n\}$. Recall that such groups are defined by the corresponding $(n-2)$-dimensional vectors, as shown in the proof of Theorem~\ref{prop:bound}. Let us denote $\T=\langle \gt_{\be_1},\dots,\gt_{\be_n}\rangle$ and $\T'=\langle \gt_{\be_1}',\dots,\gt_{\be_n}'\rangle$, whose matrices are respectively individuated by the vectors  
\[{B}=\left(b^{(1)}_{2,1},\dots,b^{(1)}_{2,n-2}\right) \text{ and } {B}'=\left(b'^{(1)}_{2,1},\dots,b'^{(1)}_{2,n-2}\right).\]

Let us assume first that $B$ and $B'$ have the same Hamming weight, i.e. the same number of non-zero coordinates. In this case there exists a permutation matrix $P\in(\FF_2)^{(n-2)\times (n-2)}$ such that ${B}P={B}'$. Let $P'\in(\FF_2)^{n\times n}$ be the permutation matrix defined as
\[
P'\deq \left(\begin{array}{ccccc}
1&0&0&\dots&0\\
0&1&0&\dots&0\\
0&0&& & \\
\vdots&\vdots&&\,P&\\
0&0&& & \end{array}\right).
\]
Note that when we multiply a matrix $M$ by $P'$ on the right we are permuting the last $n-2$ columns of $M$. On other hand, multiplying $M$ by $P'^{-1}$ on the left we are permuting the last $n-2$ rows of $M$.
Hence, we have 
\[
P'^{-1}\gt_{\be_i}P'=P'^{-1}\gk_{\be_i}P'\gs_{{\be_i}P'}=\gt_{\be_iP'}'=\gt_{\be_{i\pi}}'
\]
 where $\pi$ is the index permutation induced by $P'$, thus $P'^{-1}\T P'=\T'$. This implies that two groups corresponding to vectors with the same weight are conjugated. 
 
 Let us now assume that 
\[
B=(\underbrace{1,\dots,1}_i,0,\dots,0) \mbox{ and } B'=(\underbrace{1,\dots,1}_{i+1},0,\dots,0),
\]
for some $1 \leq i \leq n-3$.
Let $P\in(\FF_2)^{n\times n}$ be the matrix whose $j$-th row $P_j=\be_j$ if $j\ne i+2$ and $P_{i+2}=\be_{i+2}+\be_{i+3}$, i.e.
\[
P\deq\left(\begin{array}{cccccc}
1         & 0        &        & 0 & \dots & 0 \\
0         & 1        &        & 0 & \dots & 0\\
\vdots & \vdots &        &    &          & 0\\
0         & \dots   & 1     & 1 & \dots & 0\\
0         & \dots        & 0     & 1 & \dots &0 \\
0        & 0         &         &     & \dots  &1 \end{array}\right).
\]
Note that $P^{-1}=P$. Note also that multiplying a matrix $M$ by $P$ on the right we are updating its $(i+3)$-th column by summing up its  $(i+2)$-th and $(i+3)$-th columns. On the other hand, multiplying a matrix $M$ by $P^{-1}=P$ on the left we are updating its $(i+2)$-th row by summing up its  $(i+2)$-th and $(i+3)$-th rows. Therefore
\[
P\gt_{\be_j}P=P\gk_{\be_j}P\gs_{{\be_j}P}=\gt_{\be_j}'
\]
 for $j\ne i+2$ and
 \[
 P\gt_{(\be_{i+2}+\be_{i+3})}P=\gt_{\be_{i+2}}'.
 \]
 Notice that the group
 $$\langle \gt_{\be_1},\dots,\gt_{\be_{i+1}},\gt_{(\be_{i+2}+\be_{i+3})},\gt_{\be_{i+3}},\dots,\gt_{\be_n}\rangle $$ is exactly $\T$, as $\gt_{(\be_{i+2}+\be_{i+3})}\gt_{\be_{i+3}}=\gt_{\be_{i+2}}$. Therefore $P\T P=\T'$. 
 We have also proved that, if $B$ and $B'$ are such that the difference of their Hamming weights is one, by arguments previously used, the associated groups $\T$ and $\T'$ are conjugated in $\GL(V)$. 

To conclude, let us address the general case, i.e.\ the case of two groups obtained by two vectors $B$ and $B'$ having Hamming weight $d_1$ and $d_2$. Let us assume, without loss of generality, $d_1<d_2$. Let us define
\[
B_0\deq(\underbrace{1,\dots,1}_{d_1},0,\dots,0),\,B_1\deq(\underbrace{1,\dots,1}_{d_1+1},0,\dots,0),
\]
\[
\dots,\,B_{d_2-d_1}\deq(\underbrace{1,\dots,1}_{d_2},0,\dots,0),
\]
and denote by $\T(B_0),\T(B_1),\dots,\T(B_{d_2-d_1})$ the corresponding groups. Reasoning as above, we have that $\T$ and $\T(B_0)$ are conjugated in $\GL(V)$ since $B$ and $B_0$ have the same Hamming weight, and the same can be proved for $\T'$ and $\T(B_{d_2-d_1})$. Moreover, from a previous argument $\T(B_i)$ is conjugated in $\GL(V)$ to $\T(B_{i+1})$, for each $0\le i\le d_2-d_1-1$. Therefore, $\T$ and $\T'$ are conjugated in $\GL(V)$, which is our claim.
 \end{proof}
\subsection{Conjugacy classes in low dimension}
In this last section we will focus on spaces with low dimension, i.e. with dimension up to 6. 
From Proposition \ref{lm:forma} we obtain the following corollary.
\begin{corollary}\label{th:dim6}
If $\dim(V)\le 6$, then $T_+\subseteq \AGL(V,\circ)$ if and only if $T_\circ\subseteq \AGL(V,+)$.
\end{corollary}

The bound of the previous result is tight, as shown below.
\begin{proposition}
Let $V$ be such that $\dim(V)\ge 7$. Then there exists an elementary abelian regular subgroup $T_\circ < \AGL(V,+)$ such that $\AGL(V,\circ)$ does not contain $T_+$.
\end{proposition}
\begin{proof}
Let $n \geq 7$ be the dimension of $V$. If $n > 7$, let us decompose $V$ as $V=V_1\oplus V_2$, where 
\[
V_1\deq\Span\{ \be_1,\be_2,\be_3,\be_4,\be_5,\be_6,\be_7\}
\]
and 
\[
V_2\deq\Span\{ \be_8, \dots,\be_{n}\},
\]
otherwise we consider only $V_1$. 
Let us impose over $V_1$ the algebra structure induced by the exterior algebra over a vector space of dimension $3$, which is the one defined by
\[
\be_1\wedge \be_2=\be_4, \be_1\wedge \be_3=\be_5, \be_2\wedge \be_3=\be_6, \be_1\wedge \be_2 \wedge \be_3=\be_7,
\]
and over $V_2$ the algebra structure given by the trivial product $x*y \deq 0$ for each $x,y \in V_2$.
Hence we can define the following product over $V$:
\[
v\cdot w=(v_1+v_2)\cdot(w_1+w_2)\deq(v_1\wedge w_1+v_2*w_2)=v_1\wedge w_1,
\]
 where $v_1,w_1 \in V_1$ and $v_2,w_2 \in V_2$.
 It is easy to check that $(V,+,\cdot)$ is a commutative associative $\FF_2$-algebra such that the resulting ring is radical. From Theorem~\ref{th:somme}, such an algebra corresponds to an elementary abelian regular subgroup $T_\circ$ of $\AGL(V,+)$.
The claim follows from Lemma~\ref{lm:normalizer} and from its consequences, since $\be_1\cdot \be_2 \cdot \be_3 \neq 0$.
 \end{proof}

Let us now give a classification of all the elementary abelian regular subgroups of $\AGL(V,+)$ up to dimension $6$,  considering only the relevant cases when $2 < \dim(V) \leq 6$.
The results, summarised in Table~\ref{tab:classi}, derive from Corollary~\ref{cor:cod3} and Corollary~\ref{cor:cod2} and from some computation performed using MAGMA~\cite{bosma1997magma}.
For each admissible value of $n$, we collect in Table~\ref{tab:classi} the number of conjugacy classes of elementary abelian regular subgroups $T_\circ <\AGL(V,+)$, the number of such subgroups in each class and the corresponding dimension of $W(T_\circ)$.

{
\begin{table}[h]
\centering
\small\small
\begin{tabular}{| c | c|l|c| }
  \hline
  n & \text{\# of classes}   &\text{classes size} &$\dim(W(T_\circ))$ \\
  \hline
  \multirow{2}{*}3 & \multirow{2}{*}2 &  $|\cC_1|=1$& $3$\\
		
		&& $|\cC_2|=7$& $1$\\
   \hline
  \multirow{2}{*}4 & \multirow{2}{*}2  &$|\cC_1|=1$&$4$\\
  		
		 &&$|\cC_2|=105$&$2$\\
    \hline
    \multirow{4}{*} 5 &\multirow{4}{*} 4 & $|\cC_1|=1$& $5$\\
  		&& $|\cC_2|=1085$& $3$\\
		&& $|\cC_3|=6510$& $2$\\
		&& $|\cC_4|=868$& $1$\\
  \hline
  \multirow{8}{*} 6 & \multirow{8}{*}8 &   $|\cC_1|=1$& $6$\\
		&& $|\cC_2|=9765$& $4$\\
		&& $|\cC_3|=234360$& $3$\\
		&& $|\cC_4|=410130$& $3$\\
		&& $|\cC_5|=820260$& $2$\\
		&& $|\cC_6|=218736$& $2$\\
		&& $|\cC_7|=54684$& $2$\\
		&& $|\cC_8|=1093680$& $2$\\
  \hline
\end{tabular}
\caption{Conjugacy classes}\label{tab:classi}
\end{table}}

\section*{Acknowledgements}
Part of the results of this paper are contained in Marco Calderini's Ph.D. thesis~\cite{calderini2015boolean}, supervised by Massimiliano Sala. 
The authors gratefully thank the referee for comments and recommendations
which helped to improve the quality of the paper.


\end{document}